\documentclass[a4paper,12pt]{amsart}  

\usepackage{amsmath,amsthm,amssymb}
\usepackage{color}
\usepackage[all]{xy}
\usepackage{setspace}
\allowdisplaybreaks

\newtheorem{theorem}{Theorem}[section]
\newtheorem{proposition}[theorem]{Proposition}

\newtheorem{lemma}[theorem]{Lemma}
\newtheorem{corollary}[theorem]{Corollary}

\theoremstyle{remark}
\newtheorem{remark}[theorem]{Remark}

\numberwithin{equation}{section}

\newcommand{\R}{\mathbb R}
\newcommand{\Z}{{\mathbb Z}}

\newcommand{\C}{{\mathbb C}}
\newcommand{\Q}{{\mathbb Q}}

\title{Evaluation of Tornheim's type of double series}
\author[S.~Kadota, T.~Okamoto, K.~Tasaka]{Shin-ya Kadota, Takuya Okamoto, Koji Tasaka}
\address[Shin-ya Kadota]{Graduate School of Mathematics, Nagoya University}
\email{m13018c@math.nagoya-u.ac.jp}
\address[Takuya Okamoto]{Department of Mathematics, College of Liberal Arts and Sciences, Kitasato University}
\email{takuyaok@kitasato-u.ac.jp}
\address[Koji Tasaka]{Department of Information Science and Technology, Aichi Prefectural University}
\email{tasaka@ist.aichi-pu.ac.jp}

\date{}

\begin{document}
\maketitle
\begin{abstract}
We examine values of certain Tornheim's type of double series with odd weight. 
As a result, an affirmative answer to a conjecture about the parity theorem for the zeta function of the root system of the exceptional Lie algebra $G_2$, proposed by Komori, Matsumoto and Tsumura, is given.
\end{abstract}

\section{Introduction and main theorem}

For integers $a,b,k_1,k_2,k_3\ge1$, let
\[ \zeta_{a,b}(k_1,k_2,k_3):=\sum_{m,n>0} \frac{1}{m^{k_1}n^{k_2}(am+bn)^{k_3}}.\]
This series, which converges absolutely and gives a real number, was first introduced by the second author \cite{Okamoto1} in the study of evaluations of special values of the zeta functions of root systems associated with $A_2, B_2$ and $G_2$.
Since Tornheim \cite{Tornheim} first studied the value $\zeta_{1,1}(k_1,k_2,k_3)$, we call the value $\zeta_{a,b}(k_1,k_2,k_3)$ Tornheim's type of double series.
The purpose of this paper is to express $\zeta_{a,b}(k_1,k_2,k_3)$ with $k_1+k_2+k_3$ odd as $\Q$-linear combinations of two products of certain zeta values. 
As a prototype, we have in mind the analogous story for the parity theorem for multiple zeta values \cite[Corollary 8]{IKZ} and \cite{Tsumura} and for Tornheim's series \cite[Theorem 2]{HWZ}.
For example, the identity
\[ \zeta_{1,1}(1,1,3)=4\zeta(5)-2\zeta(2)\zeta(3)\]
is well-known.
Similar studies have been done in many articles \cite{Nakamura,SubbaraoSitaramachandrarao,Tsumura1,Tsumura2,Tsumura4,ZhouCaiBradley} (see also \cite{Okamoto}).
We will generalize the above expression to the value $\zeta_{a,b}(k_1,k_2,k_3)$ with $k_1+k_2+k_3$ odd.
As a consequence, we give an affirmative answer to a conjecture about special values of the zeta function of the root system of $G_2$, which was proposed by Komori, Matsumoto and Tsumura \cite[Eq.~(7.1)]{KMT5}.

We now state our main result.
We use the Clausen-type functions defined for a positive integer $k\ge2$ and $x\in\R$ by
\begin{equation}\label{eq1}
\begin{aligned}
C_k(x) &:= {\rm Re}\ Li_k(e^{2\pi ix}) = \sum_{m>0} \frac{\cos (2\pi mx)}{m^k},\\
S_k(x) &:= {\rm Im}\ Li_k(e^{2\pi ix}) = \sum_{m>0} \frac{\sin (2\pi mx)}{m^k},
\end{aligned}
\end{equation}
where $ Li_k(z)$ is the polylogarithm $\sum_{m>0} \frac{z^m}{m^k}$.
Note that $C_k(x)$ equals the Riemann zeta value $\zeta(k):=\sum_{m>0}\frac{1}{m^k}$ when $x\in \Z$, and $S_k(x)$ is $0$ when $x\in \frac{1}{2}\Z$.

\begin{theorem}\label{1_1}
For positive integers $N,a,b,k,k_1,k_2,k_3$ with $N={\rm lcm}(a,b)$ and $k=k_1+k_2+k_3$ odd, the value $ \zeta_{a,b}(k_1,k_2,k_3)$ can be expressed as $\Q$-linear combinations of $\pi^{2n}C_{k-2n}(\frac{d}{N})$ and $\pi^{2n+1}S_{k-2n-1} (\frac{d}{N})$ for $0\le n \le \frac{k-3}{2}$ and $ d\in \Z/N\Z$.
\end{theorem}

Theorem \ref{1_1} will be proved in Section 4 by using the generating functions.
This leads to a recipe for giving a formula for the $\Q$-linear combination in Theorem \ref{1_1}.
More precisely, one can deduce an explicit formula from Corollary \ref{2_2} and Propositions \ref{2_3}, \ref{3_4} and \ref{3_5}, but it might be much complicated (we do not develop the explicit formulas in this paper).
As an example of a simple identity, we have
\begin{equation}\label{eq1_1}
 \zeta_{1,3}(1,1,3) = \frac{1}{81} \left(367\zeta(5)-19\pi^2\zeta(3)-27 \pi S_4(\tfrac13) -4 \pi^3 S_2(\tfrac13)\right).
\end{equation}
We apply Theorem \ref{1_1} to proving the conjecture suggested by Komori, Matsumoto and Tsumura \cite[Eq.~(7.1)]{KMT5}.
This will be described in Section 5.

It is worth mentioning that since the value $\zeta_{a,b}(k_1,k_2,k_3)$ can be expressed as $\Q$-linear combinations of double polylogarithms
\begin{equation}\label{eq1_2}
Li_{k_1,k_2}(z_1,z_2)=\sum_{0<m<n} \frac{z_1^m z_2^n}{m^{k_1}n^{k_2}},
\end{equation}
Theorem \ref{1_1} might be proved by the parity theorem for double polylogarithms obtained by Panzer \cite{Erik} and Nakamura \cite{Nakamura}, which is illustrated in Remark \ref{4_1}.
In this paper, we however do not use their result to prove Theorem \ref{1_1}, since we want to keep this paper self-contained.

The contents of this paper is as follows.
In Section 2, we give an integral representation of the generating function of the values $\zeta_{a,b}(k_1,k_2,k_3)$ for any integers $a,b\ge1$.
In Section 3, the integral is computed.
Section 4 gives a proof of Theorem 1.1.
In Section 5, we recall the question \cite[Eq.~(7.1)]{KMT5} and give an affirmative answer to this.

\subsection*{Acknowledgement}
This work was partially supported by JSPS KAKENHI Grant Numbers 15K17517 and 16H07115.
The authors are grateful to Kohji Matsumoto, Takashi Nakamura and Hirofumi Tsumura for initial advice and many useful comments.


\section{Integral representation}
In this section, we give an integral representation of the generating function of the values $\zeta_{a,b}(k_1,k_2,k_3)$ for any integers $a,b\ge1$.
The integral representation of the value $\zeta_{a,b}(k_1,k_2,k_3)$ was first given by the second author \cite[Theorem 4.4]{Okamoto1}, following the method used by Zagier. 
We recall it briefly.

For an integer $k\ge0$, the Bernoulli polynomial $B_k(x)$ of order $k$ is defined by
\[\sum_{k\ge0} B_k(x)\frac{t^k}{k!} = \frac{te^{xt}}{e^t-1}.\]
The polynomial $B_k(x)$ admits the following expression (see \cite[Theorem 4.11]{AIK}): for $k\ge1$ and $x\in \R$ ($x\in \R-\Z$, if $k=1$)
\[B_k(x-[x]) = \begin{cases}-2i \dfrac{k!}{(2\pi i)^k}\displaystyle\sum_{m>0} \dfrac{\sin(2\pi m x)}{m^k} & k\ge1:{\rm odd},\\ 
-2  \dfrac{k!}{(2\pi i)^k}\displaystyle\sum_{m>0} \dfrac{\cos(2\pi m x)}{m^k} & k\ge2:{\rm even},\end{cases}\]
where $i=\sqrt{-1}$ and the summation $\displaystyle\sum_{m>0}$ is regarded as $\displaystyle\lim_{N\rightarrow \infty} \sum_{N>m>0}$ when $k=1$ (this ensures convergence).
We define the modified (generalized) Clausen function for $k\ge1$ and $x\in \R$ ($x\in \R-\Z$, if $k=1$) by
\[ Cl_k(x-[x]) =  \begin{cases}-\dfrac{k!}{(2\pi i)^{k-1}}\displaystyle\sum_{m>0} \dfrac{\cos(2\pi m x)}{m^k} & k\ge1:{\rm odd},\\
-i \dfrac{k!}{(2\pi i)^{k-1}}\displaystyle\sum_{m>0} \dfrac{\sin(2\pi m x)}{m^k} & k\ge2:{\rm even}.\end{cases}\]
With this, for $k\ge1$ and $x\in\R$ ($x\in \R-\Z$ if $k=1$), the polylogarithm $ Li_k(e^{2\pi i x}) $ can be written in the form
\begin{equation}\label{eq2_1}
 Li_k(e^{2\pi i x})  =- \frac{(2\pi i)^{k-1}}{ k!}\left( Cl_k(x-[x]) + \pi i B_k(x-[x]) \right).
 \end{equation}

We introduce formal generating functions.
For $x\in \R-\Z$, let
\[ \beta(x;t):=\sum_{k>0} \frac{B_k(x-[x])t^k}{k!} \quad {\rm and} \quad \gamma(x;t):=\sum_{k>0} \frac{Cl_k(x-[x])t^k}{k!}.\]

\begin{proposition}\label{2_1}
For integers $a,b\ge1$, we have
\begin{align*}
&\sum_{k_1,k_2,k_3>0} \zeta_{a,b}(k_1,k_2,k_3) t_1^{k_1}t_2^{k_2}t_3^{k_3}\\
&=- \frac{1}{4\pi i}  \int_0^1 \big( \gamma(ax;2\pi it_1)\beta(bx;2\pi it_2) + \beta(ax;2\pi i t_1) \gamma(bx;2\pi it_2)\big) \beta(x;-2\pi it_3)dx \\
&+ \frac{1}{4\pi^2} \int_0^1 \big(\gamma(ax;2\pi it_1)\gamma(bx;2\pi it_2)-\pi^2 \beta(ax;2\pi it_1)\beta(bx;2\pi it_2)\big) \beta(x;-2\pi i t_3)dx .
\end{align*}
\end{proposition}

\begin{proof}
When $k_1,k_2,k_3\ge2$, it follows
\begin{align*}
\int_0^1 Li_{k_1}\big(e^{2\pi i ax}\big)Li_{k_2}\big(e^{2\pi i bx}\big)\overline{Li_{k_3}\big(e^{2\pi ix})}dx  &=\int_0^1 \sum_{m,n,l>0}\frac{e^{2\pi i max}e^{2\pi inbx}e^{-2\pi ilx}}{m^{k_1}n^{k_2}l^{k_3}}dx\\
&=\sum_{m,n,l>0}\frac{1}{m^{k_1}n^{k_2}l^{k_3}}\int_0^1e^{2\pi ix (am+bn-l)}dx\\
&= \zeta_{a,b}(k_1,k_2,k_3),
 \end{align*}
where $\overline{Li_{k_3}\big(e^{2\pi ix})}$ stands for complex conjugate of $Li_{k_3}\big(e^{2\pi ix})$.
For $k_1,k_2,k_3\ge1$, the above equality is justified by replacing the integral $\displaystyle\int_0^1$ with 
\begin{equation}\label{eq11}
\lim_{\varepsilon\rightarrow 0}\sum_{j=1}^{{\rm lcm}(a,b)} \int_{\frac{j-1}{{\rm lcm}(a,b)}+\varepsilon}^{\frac{j}{{\rm lcm}(a,b)}-\varepsilon},
\end{equation}
where ${\rm lcm}(a,b)$ is the least common multiple of $a$ and $b$ (see \cite[Theorem 4.4]{Okamoto1} for the detail).
Letting $ Li(x;t):=\displaystyle\sum_{k>0} Li_k(e^{2\pi i x})t^k$, we therefore obtain
\begin{equation}\label{eq2_2}
\sum_{k_1,k_2,k_3>0} \zeta_{a,b}(k_1,k_2,k_3) t_1^{k_1}t_2^{k_2}t_3^{k_3}= \int_0^1 Li(ax;t_1)Li(bx;t_2)\overline{Li(x;t_3)}dx.
\end{equation}
Furthermore, the generating function of $Li_k(e^{2\pi ix})$ with $x\in \R-\Z$ can be written in the form
\begin{equation}\label{eq2_3}
Li(x;t) = -\frac{1}{2\pi i }\left( \gamma(x;2\pi i t)+\pi i \beta(x;2\pi it)\right),
\end{equation}
and hence, the right-hand side of \eqref{eq2_2} is equal to
\begin{equation}\label{eq2_4}
\begin{aligned}
&\frac{1}{(2\pi i )^3}\int_0^1 \big(\gamma(ax;2\pi i t_1)+\pi i \beta(ax;2\pi it_1)\big)\\
&\times \big(\gamma(bx;2\pi i t_2)+\pi i \beta(bx;2\pi it_2)\big)\big(\gamma(x;-2\pi i t_3)-\pi i \beta(x;-2\pi it_3)\big)dx.
\end{aligned}
\end{equation}
We note that, similarly to \eqref{eq2_2}, one obtains the relation
\begin{align*}
\int_0^1 Li(ax;t_1)Li(bx;t_2)Li(x;-t_3)dx=0,
\end{align*}
and substituting \eqref{eq2_3} to the above identity, one has
\begin{align*}
& \int_0^1 \big(\gamma(ax;2\pi i t_1)+\pi i \beta(ax;2\pi it_1)\big)\big(\gamma(bx;2\pi i t_2)+\pi i \beta(bx;2\pi it_2)\big)\gamma(x;-2\pi i t_3)dx\\
&=- \pi i \int_0^1 \big(\gamma(ax;2\pi i t_1)+\pi i \beta(ax;2\pi it_1)\big)\big(\gamma(bx;2\pi i t_2)+\pi i \beta(bx;2\pi it_2)\big)\beta(x;-2\pi it_3)dx.
\end{align*}
With this, \eqref{eq2_4} is reduced to
\begin{align*}
-\frac{1}{(2\pi i )^2}\int_0^1 \big(\gamma(ax;2\pi i t_1)+\pi i \beta(ax;2\pi it_1)\big) \big(\gamma(bx;2\pi i t_2)+\pi i \beta(bx;2\pi it_2)\big)\beta(x;-2\pi it_3)dx,
\end{align*}
which completes the proof.
\end{proof}

The coefficient of $t^k$ in $\gamma(x;2\pi i t)$ (resp. $\beta(x;2\pi it)$) is a real-valued function, if $k$ is even, and a real-valued function times $i=\sqrt{-1}$, if $k$ is odd. 
Thus, comparing the coefficient of both sides, we have the following corollary.
For simplicity, for integers $a,b\ge1$ we let
\begin{equation}\label{eq111}
 F_{a,b}(t_1,t_2,t_3) :=  \int_0^1 \gamma(ax;t_1)\beta(bx;t_2) \beta(x;-t_3)dx,
\end{equation}
where the integral is defined formally by term-by-term integration and by \eqref{eq11}.

\begin{corollary}\label{2_2}
One has
\begin{align*}
&\sum_{\substack{k_1,k_2,k_3>0\\k_1+k_2+k_3:{\rm odd}}} \zeta_{a,b} (k_1,k_2,k_3)t_1^{k_1}t_2^{k_2}t_3^{k_3}\\
&=- \frac{1}{4\pi i} F_{a,b}(2\pi i t_1,2\pi it_2,2\pi it_3) -  \frac{1}{4\pi i} F_{b,a}(2\pi i t_2,2\pi it_1,2\pi it_3).
\end{align*}
\end{corollary}

Remark that, using the same method, one can give an integral expression of the generating function of the Riemann zeta values, which will be used later.

\begin{proposition}\label{2_3}
For integers $a,b\ge1$, we have
\begin{equation}\label{eq2_5}
\frac{1}{2\pi i} \int_0^1 \gamma(ax;2\pi it_1) \beta(bx;-2\pi it_2)dx=\sum_{\substack{r,s>0\\r+s:{\rm odd}}}\frac{{\rm gcd}(a,b)^{r+s}}{a^sb^r} \zeta(r+s) t_1^{r}t_2^{s}.
\end{equation}
\end{proposition}

\begin{proof}
For any integers $a,b\ge1$, it follows
\[ \int_0^1 Li(ax;t_1)\overline{Li(bx;t_2)}dx=\sum_{r,s>0}\frac{{\rm gcd}(a,b)^{r+s}}{a^sb^r} \zeta(r+s) t_1^{r}t_2^{s}.\]
By the relation $\displaystyle\int_0^1 Li(ax;t_1)Li(bx;-t_2)dx=0$ and \eqref{eq2_3}, the left-hand side of the above equation can be reduced to 
\begin{equation*}
\frac{1}{2\pi i} \int_0^1 \big( \gamma(ax;2\pi it_1)+\pi i \beta(ax;2\pi it_1)\big) \beta(bx;-2\pi it_2)dx.
\end{equation*}
Comparing the coefficients of $t_1^rt_2^s$, we complete the proof.
\end{proof}

\section{Evaluation of integrals}
In this section, we compute the integral $F_{a,b}(t_1,t_2,t_3)$.

We denote the generating function of the Bernoulli polynomials by $\beta_0(x;t)$:
\[\beta_0(x;t):= \frac{te^{xt}}{e^t-1}=\sum_{k\ge0} B_k(x)\frac{t^k}{k!}.\]
For integers $b,c\ge1$, we set 
\begin{align*}
\alpha_{b} (t_1,t_2)&:=\beta_0(0;t_1)\beta_0(0;-t_2)\frac{e^{bt_1-t_2}-1}{bt_1-t_2},\\
\widetilde{\alpha}_{b,c}(t_1,t_2) &:= -t_1e^{-ct_1} \beta_0(0;-t_2) \frac{e^{bt_1-t_2}-1}{bt_1-t_2},
\end{align*}
which are elements in the formal power series ring $\Q[[t_1,t_2]]$.

\begin{lemma}\label{3_1}
For any integers $b,d\ge1$, we have
\[ e^{-dt_1}\alpha_b(t_1,t_2) =\alpha_b (t_1,t_2)+ \sum_{c=1}^{d} \widetilde{\alpha}_{b,c} (t_1,t_2).\]
\end{lemma}

\begin{proof}
By the relation $B_k(x)=B_k(x+1)-kx^{k-1}$ for $k\in \Z_{\ge0}$ (see \cite[Proposition 4.9 (2)]{AIK}), we have $\beta_0(x;t)=\beta_0(x+1;t)-te^{xt}$.
Using this formula with $x=-d,-d+1,\ldots,1$ repeatedly, one gets
\[ \beta_0(-d;t)=\beta_0(-d+1;t)-te^{-dt}=\cdots =\beta_0(0;t)-t\sum_{c=1}^d e^{-ct}.\]
Hence, we obtain
\begin{align*}
e^{-dt_1}\alpha_b(t_1,t_2) &= \beta_0(-d;t_1)\beta_0(0;-t_2) \frac{e^{bt_1-t_2}-1}{bt_1-t_2}\\
&=\alpha_b(t_1,t_2)-t_1 \sum_{c=1}^d e^{-ct_1}\beta_0(0;-t_2) \frac{e^{bt_1-t_2}-1}{bt_1-t_2}\\
&=\alpha_b(t_1,t_2)+ \sum_{c=1}^d\widetilde{\alpha}_{b,c}(t_1,t_2),
\end{align*}
which completes the proof.
\end{proof}

\begin{remark}\label{3_2}
Let us denote by $A_b(r,s)$ (resp.~$\widetilde{A}_{b,c}(r,s)$) the coefficient of $t_1^rt_2^s$ in $\alpha_b(t_1,t_2)$ (resp.~in $\widetilde{\alpha}_{b,c}(t_1,t_2)$).
Then, we have
\[A_b(r,s) = \sum_{\substack{p_1+q_1=r\\ p_2+q_2=s\\ p_1,p_2,q_1,q_2\ge0}} \frac{(-1)^{q_2+p_2} b^{p_1} B_{q_1}B_{q_2}}{p_1!p_2!q_1!q_2!(p_1+p_2+1)}\]
and 
\[\widetilde{A}_{b,c}(r,s) =  \sum_{\substack{p_1+q_1=r\\ p_2+q_2=s\\ p_1,p_2,q_2\ge0\\q_1\ge1}} \frac{(-1)^{q_1+q_2+p_2} c^{q_1-1} b^{p_1} B_{q_2}}{p_1!(q_1-1)!p_2!q_2!(p_1+p_2+1)},\]
where $B_k=B_k(1)=(-1)^kB_k(0)$ is the $k$-th Bernoulli number.
We note that since $\widetilde{\alpha}_{b,c}(t_1,t_2)\in t_1\Q[[t_1,t_2]]$, we have $\widetilde{A}_{b,c}(0,s)=0$ for any $s\in \Z_{\ge0}$.
\end{remark}

\begin{lemma}\label{3_3}
Let $b,d$ be positive integers with $d \in \{0,1,\ldots,b-1\}$.
Then, for $x\in (\frac{d}{b},\frac{d+1}{b})$, we have
\[ \beta(bx;t_1)\beta(x;-t_2) = e^{-dt_1}\alpha_b(t_1,t_2)  \beta_0(x;bt_1-t_2) - \beta(bx;t_1)-\beta(x;-t_2)-1,\]
where we recall $\beta(x;t)=\displaystyle\sum_{k>0} \dfrac{B_k(x-[x])}{k!}t^k$.
\end{lemma}

\begin{proof}
Since $bx-[bx]=bx-d$ when $x\in (\frac{d}{b},\frac{d+1}{b})$, one has
\begin{align*}
&\big( \beta(bx;t_1)+1\big)\big(\beta(x;-t_2)+1\big) = \frac{t_1e^{(bx-d)t_1}}{e^{t_1}-1} \frac{-t_2 e^{-xt_2}}{e^{-t_2}-1}\\
&=e^{-dt_1} \frac{t_1}{e^{t_1}-1} \frac{-t_2}{e^{-t_2}-1} e^{(bt_1-t_2)x}\\
&=e^{-dt_1} \beta_0(0;t_1)\beta_0(0;-t_2) \frac{e^{bt_1-t_2}-1}{bt_1-t_2} \frac{(bt_1-t_2)e^{(bt_1-t_2)x}}{e^{bt_1-t_2}-1}\\
&=e^{-dt_1}\alpha_b(t_1,t_2)\beta_0(x;bt_1-t_2),
\end{align*}
from which the statement follows.
\end{proof}

\begin{proposition}\label{3_4}
For any integers $a,b\ge1$, we have
\begin{equation}\label{eq3_1}
\begin{aligned}
F_{a,b}(t_1,t_2,t_3) &= \alpha_b(t_2,t_3) \int_0^1 \gamma(ax;t_1)\beta_0(x;bt_2-t_3)dx \\
&+\sum_{c=1}^{b-1} \widetilde{\alpha}_{b,c}(t_2,t_3) \int_{\frac{c}{b}}^1 \gamma(ax;t_1)\beta_0(x;bt_2-t_3)dx\\
&-\int_0^1 \gamma(ax;t_1)\big( \beta(bx;t_2)+\beta(x;-t_3)\big) dx.
\end{aligned}
\end{equation}
\end{proposition}

\begin{proof}
Splitting the integral $\displaystyle\int_0^1=\displaystyle\sum_{d=0}^{b-1} \displaystyle\int_{\frac{d}{b}}^{\frac{d+1}{b}} $ in the definition of $F_{a,b}$ (see \eqref{eq111}) and then using Lemma \ref{3_3}, we have
\begin{align*}
F_{a,b}(t_1,t_2,t_3)&=\sum_{d=0}^{b-1} \int_{\frac{d}{b}}^{\frac{d+1}{b}} \gamma(ax;t_1)\beta(bx;t_2)\beta(x;-t_3)dx\\
&=\sum_{d=0}^{b-1} e^{-dt_2}\alpha_b(t_2,t_3) \int_{\frac{d}{b}}^{\frac{d+1}{b}} \gamma(ax;t_1) \beta_0(x;bt_2-t_3)dx\\
&-\sum_{d=0}^{b-1}\int_{\frac{d}{b}}^{\frac{d+1}{b}} \gamma(ax;t_1)\big( \beta(bx;t_2)+\beta(x;-t_3)+1\big)dx\\
&=\sum_{d=0}^{b-1} \left( \alpha_b(t_2,t_3) + \sum_{c=1}^d \widetilde{\alpha}_{b,c}(t_2,t_3) \right) \int_{\frac{d}{b}}^{\frac{d+1}{b}} \gamma(ax;t_1) \beta_0(x;bt_2-t_3)dx\\
&-\int_0^1 \gamma(ax;t_1)\big( \beta(bx;t_2)+\beta(x;-t_3)+1\big)dx,
\end{align*}
where for the last equality we have used Lemma \ref{3_1}.
By the relation $\displaystyle\int_0^1 Li(ax;t)dx=0$, we have $\displaystyle\int_0^1 \gamma(ax;t_1)dx=0$.
Hence, the statement follows from and the interchange of order of summation $\displaystyle\sum_{d=1}^{b-1} \sum_{c=1}^{d} =\sum_{c=1}^{b-1} \sum_{d=c}^{b-1}$.
\end{proof}

We now deal with the integral of the second term of the right-hand side of \eqref{eq3_1}.

\begin{proposition}\label{3_5}
For any integers $a,b\ge1$ and $c\in \{0,1,\ldots,b-1\}$, we have
\begin{align*}
&\frac{1}{2\pi i} \int_{\frac{c}{b}}^1 \gamma(ax;2\pi it_1) \beta_0(x;2\pi i(bt_2-t_3))dx\\
&=-i\sum_{\substack{s\ge1\\p,q\ge 0\\p+s:{\rm odd}}} \frac{(-1)^{s}(2\pi i)^{q-1}}{q!a^{s}} S_{p+s+1}(\tfrac{ac}{b}) B_q(\tfrac{c}{b})  t_1^{p+1}(bt_2-t_3)^{q+s-1}\\
&+\sum_{\substack{s\ge1\\p,q\ge 0\\p+s:{\rm even}}} \frac{(-1)^{s}(2\pi i)^{q-1}}{q!a^{s}}\left( \zeta(p+s+1)B_q-C_{p+s+1}(\tfrac{ac}{b} )B_q(\tfrac{c}{b}) \right) t_1^{p+1}(bt_2-t_3)^{q+s-1},
\end{align*}
where $S_n(x)$ and $C_n(x)$ are defined in \eqref{eq1}.
\end{proposition}

\begin{proof}
For an integer $s\ge1$, we let
\[ \gamma_s(x;t) =\sum_{k\ge s} \frac{Cl_k(x-[x])}{k!}t^k.\]
It is easily seen that for any integer $s\ge2$ we have 
\[ \frac{d}{dx} \gamma_s(ax;t) = at\gamma_{s-1}(ax;t) \quad \mbox{ and} \quad \frac{d}{dx} \beta_0(x;t) = t\beta_0(x;t).\]
By repeated use of the integration by parts and noting that $\gamma_1(x;t)=\gamma(x;t)$, we have
\begin{align*}
&\int_{\frac{c}{b}}^1 \gamma(ax;2\pi it_1) \beta_0(x;2\pi i(bt_2-t_3)) dx\\
&=\sum_{s\ge2}\frac{(-2\pi i (bt_2-t_3))^{s-2}}{(2\pi iat_1)^{s-1}} \left[ \gamma_s (ax;2\pi it_1) \beta_0(x;2\pi i (bt_2-t_3) \right]_{\frac{c}{b}}^1\\
&=\sum_{\substack{s\ge2\\p\ge s\\q\ge0}} \frac{(-1)^s(2\pi i)^{p+q-1}}{p!q!a^{s-1}}\left[ Cl_p(ax-[ax]) B_q(x)\right]_{\frac{c}{b}}^1 t_1^{p-s+1}(bt_2-t_3)^{q+s-2}\\
&=\sum_{\substack{s\ge1\\p,q\ge 0}} \frac{(-1)^{s+1}(2\pi i)^{p+q+s}}{(p+s+1)!q!a^{s}}\left[ Cl_{p+s+1}(ax-[ax]) B_q(x)\right]_{\frac{c}{b}}^1 t_1^{p+1}(bt_2-t_3)^{q+s-1}.
\end{align*}
By definition, for any $x\in \Q$ and $k\ge2$ we have 
\[ Cl_k(x-[x])= \begin{cases} -\dfrac{k!}{(2\pi i)^{k-1}} C_k(x) & k :{\rm odd}, \\ -i\dfrac{k!}{(2\pi i)^{k-1}} S_k(x) & k:{\rm even},\end{cases} \]
and hence, the above last line is computed as follows:
\begin{align*}
&i\sum_{\substack{s\ge1\\p,q\ge 0\\p+s:{\rm odd}}} \frac{(-1)^{s}(2\pi i)^{q}}{q!a^{s}}\left( S_{p+s+1}(a)B_q(1)-S_{p+s+1}(\tfrac{ac}{b}) B_q(\tfrac{c}{b}) \right) t_1^{p+1}(bt_2-t_3)^{q+s-1}\\
&+\sum_{\substack{s\ge1\\p,q\ge 0\\p+s:{\rm even}}} \frac{(-1)^{s}(2\pi i)^{q}}{q!a^{s}}\left( C_{p+s+1}(a)B_q(1)-C_{p+s+1}(\tfrac{ac}{b} )B_q(\tfrac{c}{b}) \right) t_1^{p+1}(bt_2-t_3)^{q+s-1},
\end{align*}
which completes the proof.
\end{proof}

\section{Proof of Theorem \ref{1_1}}
We can now complete the proof of Theorem \ref{1_1} as follows.

\begin{proof}[Proof of Theorem \ref{1_1}]
We consider only the real part of the coefficient of $t_1^{k_1}t_2^{k_2}t_3^{k_3}$ in the generating function $\frac{1}{2\pi i}F_{a,b}(2\pi it_1,2\pi it_2,2\pi it_3)$ for positive integers $k,k_1,k_2,k_3$ with $k=k_1+k_2+k_3$ odd.
By \eqref{eq3_1} with $t_j\rightarrow 2\pi i t_j$, we have
\begin{equation}\label{eq4_1}
\begin{aligned}
&\frac{1}{2\pi i} F_{a,b}(2\pi it_1,2\pi it_2,2\pi it_3) \\
&= \alpha_b(2\pi it_2,2\pi it_3) \times \frac{1}{2\pi i}\int_0^1 \gamma(ax;2\pi it_1)\beta_0(x;-2\pi i(t_3-bt_2))dx \\
&+\sum_{c=1}^{b-1} \widetilde{\alpha}_{b,c}(2\pi it_2,2\pi it_3)\times \frac{1}{2\pi i} \int_{\frac{c}{b}}^1 \gamma(ax;2\pi it_1)\beta_0(x;2\pi i(bt_2-t_3))dx\\
&-\frac{1}{2\pi i}\int_0^1 \gamma(ax;2\pi it_1)\big( \beta(bx;-2\pi i(-t_2))+\beta(x;-2\pi it_3)\big) dx.
\end{aligned}
\end{equation}
It follows from \eqref{eq2_5} that the real part of the coefficient of $t_1^{k_1}t_2^{k_2}t_3^{k_3}$ in the first and last term of the right-hand side of \eqref{eq4_1} can be expressed as $\Q$-linear combinations of $\pi^{2n}\zeta(k-2n)$ with $0\le n \le \frac{k-3}{2}$.
For the second term, using Proposition \ref{3_5} (see also Remark \ref{3_2}), we have
\begin{equation}\label{eq4_2}
\begin{aligned}
&\widetilde{\alpha}_{b,c}(2\pi it_2,2\pi it_3)\times \frac{1}{2\pi i} \int_{\frac{c}{b}}^1 \gamma(ax;2\pi it_1) \beta_0(x;2\pi i(bt_2-t_3))dx\\
&=-i\sum_{\substack{n_2\ge1\\n_3\ge0}}\sum_{\substack{s\ge1\\p,q\ge 0\\p+s:{\rm odd}}} \frac{(-1)^{s}\widetilde{A}_{b,c}(n_2,n_3)}{q!a^{s}}(2\pi i)^{n_2+n_3+q-1} S_{p+s+1}(\tfrac{ac}{b}) B_q(\tfrac{c}{b})t_1^{p+1}(bt_2-t_3)^{q+s-1}t_2^{n_2}t_3^{n_3}\\
&+\sum_{\substack{n_2\ge1\\n_3\ge0}} \sum_{\substack{s\ge1\\p,q\ge 0\\p+s:{\rm even}}} \frac{(-1)^{s}\widetilde{A}_{b,c}(n_2,n_3)}{q!a^{s}}(2\pi i)^{n_2+n_3+q-1} \left( \zeta(p+s+1)B_q-C_{p+s+1}(\tfrac{ac}{b} )B_q(\tfrac{c}{b}) \right)\\
&\times  t_1^{p+1}(bt_2-t_3)^{q+s-1}t_2^{n_2}t_3^{n_3},
\end{aligned}
\end{equation}
where we note that in the above both summations, $p+s+1$ runs over integers greater than 1.
Since for any $x\in \Q$ and $k\ge0$ we have $B_k(x)\in\Q$, the real part of the coefficient of $t_1^{k_1}t_2^{k_2}t_3^{k_3}$ in the first term (resp. the second term) of the right-hand side of \eqref{eq4_2} is a $\Q$-linear combination of $\pi^{2n+1}S_{k-2n-1}(\frac{ac}{b})$ with $0\le n \le \frac{k-3}{2}$ (resp. $\pi^{2n}C_{k-2n}(\frac{ac}{b})$ and $\pi^{2n}\zeta(k-2n)$ with $0\le n \le \frac{k-3}{2}$).
We therefore find that the real part of the coefficient of $t_1^{k_1}t_2^{k_2}t_3^{k_3}$ in the generating function $\frac{1}{2\pi i}F_{a,b}(2\pi it_1,2\pi it_2,2\pi it_3)$ can be expressed as $\Q$-linear combinations of $\pi^{2n+1}S_{k-2n-1}(\frac{ac}{b})$ and $\pi^{2n}C_{k-2n}(\frac{ac}{b})$ with $0\le n \le \frac{k-3}{2}$ and $c\in \Z/b\Z$.
Thus by Corollary \ref{2_2} we complete the proof.
\end{proof}

\begin{remark}\label{4_1}
As mentioned in the introduction, the value $\zeta_{a,b}(k_1,k_2,k_3)$ is expressible as $\Q$-linear combinations of double polylogarithms $Li_{r,s}(z_1,z_2)$ defined in \eqref{eq1_2}, where the expression is obtained from the partial fractional decomposition
\[  \frac{1}{x^ry^s} = \sum_{\substack{p+q=r+s\\p,q\ge1}} \frac{1}{(x+y)^p}\left(\binom{p-1}{s-1}  \frac{1}{x^{q}} + \binom{p-1}{r-1}\frac{1}{y^q} \right) \qquad  (r,s\in \Z_{\ge1}) \]
and the orthogonality relation
\[ \frac{1}{N} \sum_{n\in \Z/N\Z} \mu_N^{dn} = \begin{cases} 1 & N\mid d \\ 0 & N\nmid d\end{cases},\]
where $\mu_N=e^{2\pi i/N}$ and $d\in \Z$.
For example, one can check
\begin{equation}\label{eq4_3}
\zeta_{1,3}(1,1,3) =\sum_{u\in \Z/3\Z} Li_{1,4}(\mu_3^{-u},\mu_3^u) + \sum_{u\in \Z/3\Z} Li_{1,4}(\mu_3^u,1).
\end{equation}
From this, Theorem \ref{1_1} might be proved by the parity theorem for double polylogarithms examined in \cite[Eq.~(3.2)]{Erik}.
Although we do not proceed this in general, let us illustrate an example. 
As a special case of \cite[Eq.~(3.2)]{Erik}, one obtains
\begin{align*}
Li_{1,4}(z_1,z_2)+Li_{1,4}(z_1^{-1},z_2^{-1}) &= \sum_{n=1}^5 (-1)^{n+1} Li_n(z_1) \mathcal{B}_{5-n}(z_1z_2) - Li_1(z_1)\mathcal{B}_4(z_2) \\
&+\sum_{n=4}^5 \binom{n-1}{3} Li_n(z_2^{-1}) \mathcal{B}_{5-n}(z_1z_2) - Li_5(z_1z_2),
\end{align*}
where for each integer $k\ge0$ we set $\mathcal{B}_k(z)=\frac{(2\pi i)^k}{k!} B_k \left(\frac{1}{2} +\frac{\log(-z)}{2\pi i}\right)$.
We note that $Li_k(\mu_3^u)= C_k(\frac{u}{3})+iS_k(\frac{u}{3})$ and $\mathcal{B}_k(\mu_3)=\frac{(2\pi i)^k}{k!} B_k(\frac{1}{3})$ since $\log(-\mu_3)=-\frac{\pi i}{3}$.
With this, the above formula gives
\begin{align*}
&{\rm Re}\ (Li_{1,4}(\mu_3^{-1},\mu_3)+Li_{1,4}(\mu_3^{-2},\mu_3^2))=\frac{1}{243}\left(-843\zeta(5) +36\pi^2\zeta(3)+ 4\pi^4\log3\right),\\
&{\rm Re}\ (Li_{1,4}(\mu_3,1)+Li_{1,4}(\mu_3^2,1))=\frac{1}{243}\left(972\zeta(5)-12\pi^2\zeta(3)-4\pi^4 \log3- 81 \pi S_4(\tfrac13) -12\pi^3 S_2(\tfrac13)\right),\\
&2Li_{1,4}(1,1)=4\zeta(5) - \frac13 \pi^2 \zeta(3).
\end{align*}
where we have used $C_k(\frac13)=C_k(\frac23)=\frac{1-3^{k-1}}{2\cdot 3^{k-1}}\zeta(k)$ for $k\ge2$ and $C_1(\frac13)=C_1(\frac23)=-\frac{1}{2} \log 3$.
Substituting the above formulas to \eqref{eq4_3}, one gets \eqref{eq1_1}.
We have checked Theorem \ref{1_1} for $(a,b)=(1,3)$ and $(2,3)$ in this direction.
\end{remark}

\section{The zeta function of the root system $G_2$}

In this section, we give an affirmative answer to the question posed by Komori, Matsumoto and Tsumura \cite[Eq.~(7.1)]{KMT5}.

The zeta-function associated with the exceptional Lie algebra $G_2$ is defined for complex variables ${\bf s}=(s_1,s_2,\ldots,s_6)\in \C^6$ by
\[ \zeta({\bf s};G_2) := \sum_{m,n>0} \frac{1}{m^{s_1}n^{s_2}(m+n)^{s_3}(m+2n)^{s_4}(m+3n)^{s_5}(2m+3n)^{s_6}} .\]
The function $\zeta ({\bf s};G_2) $ was first introduced by Komori, Matsumoto and Tsumura (see \cite{KMT4,KMT5}), where they developed its analytic properties and functional relations. 
They also examined explicit evaluations of the special values of $\zeta({\bf k};G_2)$ at ${\bf k} \in \Z_{>0}^6$ (see \cite{Zhao} for ${\bf k} \in \Z_{\ge0}^6$), where we note that the series $\zeta({\bf k};G_2) $ converges absolutely for ${\bf k}\in \Z_{>0}^6$.
For example, they showed 
\[ \zeta(2,1,1,1,1,1;G_2)=- \frac{109}{1296} \zeta(7)+\frac{1}{18} \zeta(2)\zeta(5) .\]
Komori, Matsumoto and Tsumura \cite[Eq.~(7.1)]{KMT5} suggested a conjecture, which we now prove, that the value $\zeta(k_1,\ldots,k_6;G_2)$ with $k_1+\cdots+k_6$ odd lies in the polynomial ring over $\Q$ generated by $\zeta(k) \ (k\in \Z_{\ge2})$ and $L(k,\chi_3) \ (k\in \Z_{\ge1})$, where $L(s,\chi_3)$ is the Dirichlet $L$-function associated with the character $\chi_3$ defined by
\[ L(s,\chi_3) = \sum_{m>0} \frac{\chi_3(m)}{m^s}\]
and the character $\chi_3$ is determined by $\chi_3(n)=1$ if $n\equiv 1 \mod 3$, $\chi_3(n)=-1$ if $n\equiv 2\mod 3$ and $\chi_3(n)=0$ if $n\equiv 0 \mod 3$.
We remark that the second author \cite{Okamoto1} showed that the value $\zeta(k_1,\ldots,k_6;G_2)$ with $k_1+\cdots+k_6$ odd can be written in terms of $\zeta(s),L(s,\chi_3),S_r(\frac{d}{N}),C_r(\frac{d}{N})$ for $N=4,12$ and $0<d<N, \ (d,N)=1$ (see also \cite[\S7]{KMT5}).
The following theorem gives an affirmative answer to the question.

\begin{theorem}
For any integers $k,k_1,\ldots,k_6\ge1$ with $k=k_1+\cdots+k_6$ odd, the value $\zeta(k_1,\ldots,k_6;G_2)$ can be expressed as $\Q$-linear combinations of $\zeta(2n) \zeta(k-2n) \ (0\le n \le \frac{k-3}{2})$ and $L(2n+1,\chi_3)L(k-2n-1,\chi_3) \ (0\le n \le\frac{k-3}{2})$, where $\zeta(0)=-\frac{1}{2}$.
\end{theorem}
\begin{proof}
In \cite[Theorem 2.3]{Okamoto1}, the second author proved that for any integers $l_1,\ldots,l_6\ge1$, the value $\zeta(l_1,\ldots,l_6;G_2)$ can be expressed as $\Q$-linear combinations of $\zeta_{a,b}(n_1,n_2,n_3)$ with $(a,b)=(1,1),(1,2),(1,3),(2,3)$, $n_1+n_2+n_3 =l_1+\cdots+l_6$ and $n_1,n_2,n_3\in\Z_{>0}$.
As a consequence, it follows from Theorem \ref{1_1} that the value $\zeta(k_1,\ldots,k_6;G_2)$ can be written as $\Q$-linear combinations of $\pi^{2n}C_{k-2n}(\frac{d}{6})$ and $\pi^{2n+1} S_{k-2n-1}(\frac{d}{6})$ with $0\le n\le \frac{k-3}{2}$ and $d\in \Z/6\Z$.
For any $d\in \Z/6\Z$ and $k\ge2$ we have $C_k(\frac{d}{6})\in \Q\zeta(k) $ and $S_k(\frac{d}{6}) \in \Q \sqrt{3}L(k,\chi_3)$, and then the result follows from the well-known formula: $\zeta(2n)\in \Q\pi^{2n}, \ L(2n+1,\chi_3)\in \Q \sqrt{3} \pi^{2n+1}$ for any $n\in\Z_{\ge0}$ (see \cite[Theorem 9.6]{AIK}).
\end{proof}

Let us illustrate an example of the formula for $\zeta(k_1,\ldots,k_6;G_2)$.
Applying the partial fractional decomposition repeatedly to the form $(m+n)^{-k_3}(m+2n)^{-k_4}(m+3n)^{-k_5}(2m+3n)^{-k_6}$, we get
\begin{align*}
&\zeta(1,1,1,1,1,2;G_2)\\
&=\frac{1}{2} \zeta_{1,1}(5,1,1)-16\zeta_{1,2}(5,1,1)+\frac{9}{2} \zeta_{1,3}(5,1,1)+9\zeta_{2,3}(4,1,2)+18\zeta_{2,3}(5,1,1).
\end{align*}
Then, by Theorem \ref{1_1} (actually we use Corollary \ref{2_2} together with Propositions \ref{2_3}, \ref{3_4} and \ref{3_5}), we have
\begin{align*}
\zeta(1,1,1,1,1,2;G_2) &= \frac{2507}{1296}\zeta(7)-\frac{505}{648}\pi^2\zeta(5)+\frac{9}{4}\pi S_6(\tfrac{1}{3})\\
&=\frac{2507}{1296}\zeta(7)-\frac{505}{108}\zeta(2)\zeta(5)+\frac{3}{8}L(1,\chi_3) L(6,\chi_3),
\end{align*}
where $L(1,\chi_3)=\frac{\pi}{3\sqrt{3}}$.



\begin{thebibliography}{99}
\bibitem{AIK} T.~Arakawa, T.~Ibukiyama, M.~Kaneko, Bernoulli numbers and Zeta functions, {\itshape Springer Monographs in Mathematics}. Springer, Tokyo, 2014.

\bibitem{Erik} E.~Panzer, The parity theorem for multiple polylogarithms, {\itshape J.~Number Theory}, {\bf 172} (2017), 93--113.

\bibitem{HWZ} J.G.~Huard, K.S.~Williams, N.Y.~Zhang, On Tornheim's double series, {\itshape Acta Arith.}, {\bf 75} (1996), no. 2, 105--117.

\bibitem{IKZ} K.~Ihara, M.~Kaneko, D.~Zagier, Derivation and double shuffle relations for multiple zeta values, {\itshape Comp.~Math.}, {\bf 142} (2006), no. 2, 307--338.


\bibitem{KMT4} Y.~Komori, K.~Matsumoto, H.~Tsumura, On Witten multiple zeta-functions associated with semi-simple Lie algebras IV, {\itshape Glasg.~Math.~J.}, {\bf 53} (2011), no. 1, 185--206.

\bibitem{KMT5} Y.~Komori, K.~Matsumoto, H.~Tsumura, On Witten multiple zeta-function associated with semi-simple Lie algebras V, {\itshape Glasg.~Math.~J.}, {\bf 57} (2015), no.1, 107--130.

\bibitem{Nakamura} T.~Nakamura, A simple proof of the functional relation for the Lerch type Tornheim double zeta function, {\itshape Tokyo J.~Math.}, {\bf 35} (2012), no. 2, 333--337.


\bibitem{Okamoto1} T.~Okamoto, Multiple zeta values related with the zeta-function of the root system of type $A_2,B_2$ and $G_2$, {\itshape Comment.~Math.~Univ.~St.~Pauli}, {\bf 61} (2012), no. 1, 9--27.

\bibitem{Okamoto} T.~Okamoto, On alternating analogues of the Mordell-Tornheim triple zeta values, {\itshape J. Ramanujan Math. Soc.}, {\bf 28} (2013), no. 2, 247--269.

\bibitem{SubbaraoSitaramachandrarao} M.V.~Subbarao, R.~Sitaramachandrarao, On some infinite series of L. J. Mordell and their analogues, {\itshape Pacific J.~Math.}, {\bf 119} (1985), no. 1, 245--255.

\bibitem{Tornheim} L.~Tornheim, Harmonic double series, {\itshape Amer.~J.~Math.}, {\bf 72} (1950), 303--314.

\bibitem{Tsumura1} H.~Tsumura, On alternating analogues of Tornheim's double series, {\itshape Proc.~Amer.~Math.~Soc.}, {\bf 131} (2003), no. 12, 3633--3641.

\bibitem{Tsumura2} H.~Tsumura, Evaluation formulas for Tornheim's type of alternating double series, {\itshape Math.~Comp.}, {\bf 73} (2004), no. 245, 251--258.

\bibitem{Tsumura} H.~Tsumura, Combinatorial relations for Euler-Zagier sums, {\itshape Acta Arith.}, {\bf 111} (2004), no. 1. 27--42.


\bibitem{Tsumura4} H.~Tsumura, On alternating analogues of Tornheim's double series II, {\itshape Ramanujan J.}, {\bf 18} (2009), no. 1, 81--90.




\bibitem{Zhao} J.~Zhao, Multi-polylogs at twelfth roots of unity and special values of Witten multiple zeta function attached to the exceptional Lie algebra $\mathfrak{g}_2$, {\itshape J. Algebra Appl.}, {\bf 9} (2010), no. 2, 327--337.

\bibitem{ZhouCaiBradley} X.~Zhou, T.~Cai, D.M.~Bradley, Signed $q$-analogs of Tornheim's double series, {\itshape Proc.~Amer.~Math.~Soc.}, {\bf 136} (2008), no. 8, 2689--2698.

\end{thebibliography}
\end{document}